\documentclass{article}
\usepackage[reqno]{amsmath} 
\usepackage{amssymb, amsthm}
\usepackage{enumerate}
\usepackage{url} 
\usepackage{makeidx} 

\newtheoremstyle{plainsl}
	{\topsep}
	{\topsep}
	{\slshape} 
	{}
	{\normalfont\bfseries}
	{.}
	{ }
	{}

\swapnumbers

\theoremstyle{plainsl}
\newtheorem{theorem}{Theorem}[section]

\newtheorem{lemma}[theorem]{Lemma}
\newtheorem{corollary}[theorem]{Corollary}

\theoremstyle{definition}

\newcommand\cref[1]{Corollary~\ref{cor:#1}}

\newcommand\sqr[2]{{\vbox{\hrule height.#2pt
    \hbox{\vrule width.#2pt height#1pt \kern#1pt
        \vrule width.#2pt}\hrule height.#2pt}}}
\renewcommand\qed{%
	\ifmmode\eqno\sqr53
	\else\nolinebreak\ \hfill\sqr53\medbreak\fi}

\numberwithin{equation}{section}
\newcommand\ints{{\mathbb Z}}

\newcommand\fld{{\mathbb F}}

\begin{document}

\title{Strongly regular $n$-e.c.~graphs}
\author{Natalie Mullin\thanks{Research supported by NSERC.} \\
\small Department of Combinatorics and Optimization\\ \small University of Waterloo,  Waterloo, Ontario, Canada \\ \small nmullin@math.uwaterloo.ca}

\maketitle

\abstract{A result of Erd\"{o}s and R\'{e}nyi shows that for a fixed integer $n$ almost all graphs satisfy the $n$-e.c.~adjacency property. However, there are few explicit constructions of $n$-e.c. graphs for $n > 2$, and almost all known families of $n$-e.c. graphs are strongly regular graphs. In this paper we derive parameter bounds on strongly regular $n$-e.c.~graphs constructed from the point sets of partial geometries. This work generalizes bounds on $n$-e.c.~block intersection graphs of balanced incomplete block designs given by McKay and Pike. It also relates to work by Griggs, Grannel, and Forbes' determining $3$-e.c.~graphs that are block intersection graphs of Steiner triple systems. In addition to these bounds, we give examples of strongly regular graphs that contain every possible subgraph of small order but are not $n$-e.c. for $n > 2$. }

\section{Introduction}
The problem of constructing graphs with the $n$-e.c.~adjacency property has been studied in several recent papers. A graph $G$ is \textsl{$n$-existentially complete}, which we denote $n$-e.c., if and only if for disjoint subsets $A$ and $B$ of $V$ such that $|A \cup B| = n$, there is a vertex $z$ outside of $A \cup B$ that is adjacent to each vertex in $A$ and no vertex in $B$. In particular, $A$ or $B$ might be the empty set. In 1963, Erd\"{o}s and R\'{e}nyi proved that for a fixed $n$ almost all graphs are $n$-e.c.~\cite{erdos}. Despite this result, there are relatively few explicit constructions of families of $n$-e.c.~graphs for $n \geq 3$, and nearly all known examples are strongly regular graphs \cite{bonato_survey}.

In this paper we focus on strongly regular graphs that arise as point graphs of partial geometries. Such graphs are referred to as \textsl{geometric} graphs. A result due to Neumaier and Sims shows that for a fixed integer $m$, all but finitely many non-trivial strongly regular graphs with least eigenvalue $-m$ are geometric (see \cite{neumaier} and \cite{ray-chaudhuri}).  Using elementary methods we determine bounds on the parameters of $n$-e.c.~geometric graphs. As an application of these bounds, we show there is a unique $3$-e.c.~Latin square graph. This is a generalization of similar bounds given by McKay and Pike for $n$-e.c.~graphs that are block intersection graphs of balanced incomplete block designs \cite{mckay_pike}. It also relates to work by Forbes, Grannell, and Griggs determining $3$-e.c. block intersection graphs of Steiner triple systems \cite{forbes}.

If an infinite graph is $n$-e.c.~for all $n$, then the graph is called e.c., and it is isomorphic to the Rado graph. In this light, if $G$ is $n$-e.c., then $G$ is a finite analogue of the Rado graph. Accordingly, the $n$-e.c.~properties can be viewed as a deterministic measure of randomness in a graph. However, in this paper we show there is an infinite family of graphs that are not $n$-e.c.~for $n \geq 3$ but contain every possible subgraph of small order. Thus we show that the $n$-e.c.~property is significantly more restrictive than the related notion of $r$-fullness.

\section{$n$-e.c.~Graphs}
We consider a graph $G$ with vertex set $V$ and a positive integer $n$. For a specific pair of disjoint subsets $A$ and $B$ such that $|A \cup B| = n$, we refer to a vertex $z$ not in $A \cup B$ as an extension of $A$ and $B$ if $z$ is adjacent to each vertex in $A$ and no vertex in $B$. Thus if $G$ is $n$-e.c., then there is an extension to every bipartition of every $n$-set of $V$.

From this definition we see that a graph $G$ is $1$-e.c.~if and only if there are no isolated or dominating vertices. In the next section we see there is a close connection between $2$-e.c.~graph and strongly regular graphs. It is useful to note the following basic results from \cite{bonato_survey} concerning $n$-e.c.~graphs.

\begin{theorem}
\label{basicprops}
If $G$ is $n$-e.c., then the following hold.
\begin{enumerate}[(i)]
\item{\label{prop1} $G$ is $m$-e.c.~for all $m$ such that $1 \leq m \leq n$.}
\item{\label{prop2} The complement of $G$ is $n$-e.c.}
\item{\label{prop3} For each vertex $x$, the neighbourhood of $x$ induces an ($n$-1)-e.c.~graph.}
\item{\label{prop4} For each vertex $x$, every possible subgraph on $n+1$ vertices occurs as an induces a subgraph of $G$ containing $x$. \qed}
\end{enumerate}
\end{theorem}

Property (\ref{prop1}) implies that if $G$ is $n$-e.c.~for $n \geq 3$, then $G$ must be $2$-e.c.~This justifies our careful treatment of $2$-e.c.~graphs in the proceeding section. 

A graph that contains every possible subgraph of order $r$ is called an \textsl{$r$-full graph}. Thus by Property (\ref{prop4}), every $n$-e.c.~graph is $(n+1)$-full. However, not every graph that is $(n+1)$-full is $n$-e.c. As an example we investigate the family of symplectic graphs. These graphs are considered more extensively in Godsil and Royle (\cite{godsil_royle}. 

Let $\fld_2$ denote the binary field. Fix some positive integer $r$, and let $N$ denote the block diagonal matrix with $r$ blocks of the form
\begin{equation*}
\begin{pmatrix} 0&1\\ 1&0 \end{pmatrix},
\end{equation*}
whose entries are elements of $\fld_2$. The symplectic graph of order $2r$, denote $Sp(2r)$, is the graph whose vertex set are the elements of $\fld_2^{2r} \setminus \{0\}$ such that vertices $x$ and $y$ adjacent if and only if $x^T N y = 1.$ It is known that every graph on $2r-1$ vertices occurs as an induced subgraph of $Sp(2r)$ (Section 8.11 \cite{godsil_royle}). Despite this randomness property, $Sp(2r)$ is not  $n$-e.c.~for any $n > 2$. By Property (\ref{prop1}), it suffices to show that $Sp(2r)$ is not $3$-e.c.

\begin{lemma}
For every positive integer $r$, the graph $Sp(2r)$ is not $3$-e.c.
\end{lemma}
\begin{proof}
Let $x_1$ and $x_2$ denote two distinct elements of $\fld_2^{2r}$. Choose a third vertex $x_3$ such that $x_3 = x_1 + x_2$. Note that $x_3$ is nonzero and distinct from $x_1$ and $x_2$. Let $z$ denote a vertex adjacent to $x_1$ and $x_2$. This implies that $x_1^T N z = 1$ and $x_2^T N z = 1$, and so 
\begin{equation*}
x_3^T N z = (x_1+ x_2)^T N z = 0.
\end{equation*}
Thus every vertex $z$ in $Sp(2r)$ adjacent to both $x_1$ and $x_2$ is not adjacent to $x_3$, and so $x_1$, $x_2$, and $x_3$ do not have a common neighbour. \qed
\end{proof}

\section{Strongly Regular Graphs}
We recall that a strongly regular graph with parameter set $(v,k,\lambda,\mu)$ is a $k$-regular graph on $v$ vertices such every pair of adjacent vertices has $\lambda$ common neighbours and every pair of nonadjacent vertices has $\mu$ common neighbours. It is further required that every strongly regular graph contains at least one edge and at least one pair of distinct nonadjacent vertices. As we see in the following well-known lemma (see \cite{godsil_royle}, for example), three of these parameters determine the fourth.

\begin{lemma}
\label{counting_lemma}
The parameter set $(v,k,\lambda,\mu)$ of a strongly regular graph satisfies
\begin{equation*}
(v-k-1)\mu = k (k - \lambda -1).
\end{equation*}
\end{lemma}
\begin{proof}
Fix a vertex in the graph and count the edges between the set of neighbours and the set of non-neighbours of the vertex. \qed
\end{proof}

There is a straightforward correspondence between strongly regular graphs and $2$-e.c.~graphs.

\begin{theorem}
\label{srg_2ec}
A strongly regular graph $G$ with complement $\overline{G}$ is $2$-e.c.~if and only if $G$ and $\overline{G}$ are both connected and contain a triangle.
\end{theorem}
\begin{proof}
Suppose $G$ is strongly regular with parameters $(v,k,\lambda,\mu)$. Then $G$ is connected and contains a triangle if and only if $k > 0$ and $\mu > 0$.

Consider two adjacent vertices $x$ and $y$. Since $G$ is connected, we see that $x$ and $y$ have a common neighbour if and only if $G$ contains a triangle. Likewise, there is a vertex nonadjacent to both $x$ and $y$ if and only if $v > 2k - \lambda$. This holds if and only if $v-k-1 > k - \lambda -1$. However, from Lemma \ref{counting_lemma} we have $(v-k-1)\mu = k (k - \lambda -1)$, and so we see that $v > 2k - \lambda$ if and only if $k > \mu$. Finally, we note that there is a vertex adjacent to $x$ but not adjacent to $y$ if and only $k > \lambda + 1$. However, this holds for every connected strongly regular graph, since otherwise the graph would be complete. Thus we have $k > \lambda + 1$ if and only if $\mu > 0$.

Next we consider vertices $x$ and $y$ that are nonadjacent in $G$. Such $x$ and $y$ will be adjacent in the complement of $G$, denoted $\overline{G}$. It is useful to note that the parameters of $\overline{G}$ are
\begin{equation*}
(v,v-k-1,v-2k+\mu-2,v-2k+\lambda).
\end{equation*}
Note that $x$ and $y$ have a common neighbour in $\overline{G}$ if and only if $\overline{G}$ contains a triangle. Furthermore, there is a vertex not adjacent to either $x$ or $y$ in $\overline{G}$ if and only if $v > 2(v-k-1) - v + 2k - \mu + 2$, which holds if and only if $\mu > 0$. Finally, we see that there is a vertex adjacent to $x$ and not adjacent to $y$ in $\overline{G}$ if and only if $k > \mu$. \qed
\end{proof}

We note that the conditions in the previous theorem are not restrictive. The only disconnected strongly regular graphs are disjoint copies of a complete graph, and there are only seven known triangle-free strongly regular graphs with connected complements \cite{godsil_problems}.

\section{Partial Geometries}
Now we turn our attention to a special point-line incidence structure that is useful for constructing strongly regular graphs. A \textsl{partial geometry} $pg(s,t,\alpha)$ is a partial linear space with constant line size $s+1$, such that each point is on $t+1$ lines, and given a point not $p$ not on a line $L$, there are exactly $\alpha$ lines through $p$ that meet $L$. For nondegeneracy we require $s \geq 2$, $t \geq 1$, and $\alpha \geq 1$, and it is implicit in the definition that $s,t \geq \alpha-1$. The dual of a partial geometry is the point-line incidence structure obtained by switching the points and lines.

Partial geometries can be divided into four classes:
\begin{enumerate}[(1)]
\item{\label{design} A partial geometry with $s+1 = \alpha$ is a $2$-$(v,s+1,1)$ design. The dual satisfies $t+1= \alpha$ and is a called a \textsl{dual design}.}
\item{A partial geometry with $t = \alpha$ is called a \textsl{net}. The dual satisfies $s = \alpha$ and is called a \textsl{transversal design}.}
\item{A partial geometry with $\alpha = 1$ is called a \textsl{generalized quadrangle}.}
\item{If $1 < \alpha < \; \text{min}\{s,t\}$, then the partial geometry is called \textsl{proper}.}
\end{enumerate}

The point graph of $pg(s,t,\alpha)$ is a graph constructed on the points of the partial geometry whose edge set consists of unordered pairs of collinear points. We refer to such a graph as a \textsl{geometric} graph.

From this definition it follows that the point graph of nondegenerate $pg(s,t,\alpha)$ with $s+1 > \alpha$ is strongly regular with the following parameters.
\begin{eqnarray*}
\label{pg_params}
&v = (s+1)(st+\alpha)/\alpha; \; \; &k = s(t+1) \\
&\lambda = (s-1) + t(\alpha -1)\; \;  \; &\mu = (t+1)\alpha.
\end{eqnarray*}
We see that the restriction $s + 1> \alpha$ is necessary to exclude the complete graph. Moreover, $t+1 \geq \alpha$ is implicit in the definition of a partial geometry.

If $G$ is a strongly regular graph has the same parameter set as the point graph of a partial geometry, then we refer to the graph as \textsl{pseudo-geometric}.

\begin{lemma}
\label{geotrifree}
If $G$ is pseudo-geometric graph with respect to the parameters $(s,t,\alpha) \neq (3,1,2)$ with $s \geq 2$, then $\overline{G}$ contains a triangle. 
\end{lemma}
\begin{proof}
Recall that if $G$ is strongly regular graph with parameters $(v,k,a,c)$, then two adjacent vertices in the complement $\overline{G}$ have $v - 2k + \mu - 2$ common neighbours. Thus $\overline{G}$ is triangle-free if and only if $v - 2k + \mu - 2 = 0$. If $G$ is the point graph of $pg(s,t,\alpha)$, then we can express this condition in terms of $(s,t,\alpha)$ as
\begin{equation*}
(s+1)(st+\alpha)/\alpha - 2s(t+1) + \alpha(t+1) - 2 = 0.
\end{equation*}
This equation holds if and only if
\begin{equation}
\label{comptri}
(s-\alpha)^2t + (t-\alpha)s + \alpha(\alpha -1) = 0.
\end{equation}
If $t \geq \alpha$, then equality holds in (\ref{comptri}) if and only if $s=t=\alpha=1$, which contradicts our assumption that $s \geq 2$.
On the other hand, if $t = \alpha -1$, then equality holds in (\ref{comptri}) if and only if $s = 3$, $t=1$, and $\alpha = 2$. \qed
\end{proof}

We note that if $G$ is a pseudo-geometric graph corresponding to the parameters $(3,1,2)$, then its complement must be a cubic triangle-free strongly regular on 10 vertices. The Petersen graph is unique graph with these properties. Therefore the point graph of $pg(3,1,2)$ is a uniquely determined geometric graph.

Using these parameters and Theorem \ref{srg_2ec} we obtain necessary and sufficient conditions for $2$-e.c.~strongly regular graphs.

\begin{corollary}
The point graph of a nondegenerate $pg(s,t,\alpha)$ is 2-e.c. if and only if $s  \geq \alpha + 1$ and $(s,t,\alpha) \neq (3,1,2)$.
\end{corollary}
\begin{proof}
Let $G$ denote the point graph of $pg(s,t,\alpha)$, and let $\overline{G}$ denote its complement. By Theorem \ref{srg_2ec}, it suffices to show that $G$ and $\overline{G}$ are both connected and contain and triangle if and only if the the above conditions hold. By the nondegeneracy of the partial geometry, we know that $s \geq 2$ and $\alpha \geq 1$. Therefore $G$ is connected and contains a triangle.

In terms of the parameters $G$ as a strongly regular graph, we note that $v-2k-\mu-2>0$ if and only if $k \geq \mu + 1$ if and only if $s \geq \alpha + 1$. Therefore $\overline{G}$ is connected if and only if $s \geq \alpha + 1$. Finally by Lemma \ref{geotrifree}, we note that $\overline{G}$ contains a triangle if and only if $(s,t,\alpha) \neq (3,1,2)$. \qed
\end{proof}

\section{Parameter Bounds}
Now we consider $n$-e.c.~graph from partial geometries for $n \geq 3$. For these values of $n$, it will not be possible to obtain necessary and sufficient conditions for a point graph to be $n$-e.c.~in terms of the parameters of the underlying partial geometry. As we will see, there is a unique $3$-e.c.~geometric graph corresponding to a partial geometry $pg(8,2,2)$. However, there are many non-isomorphic geometric graphs with the same parameters that are not $3$-e.c.

We focus on necessary conditions, and use these conditions to narrow the list of possible $n$-e.c.~point graphs. First we introduce convenient notation for discussing the $n$-e.c.~conditions for large $n$.

Let $G$ denote a graph with vertex set $V$. For two disjoint sets $A$ and $B$ we define $\Gamma(A,B)$ to be the number of vertices in $G$ outside of $A$ and $B$ that are adjacent to each vertex in $A$ and no vertex in $B$. In terms of this notation we see that $G$ is a $n$-e.c.~if and only if for every pair of disjoint subsets $A$ and $B$ such that $|A \cup B| = n$ satisfies $\Gamma(A,B) > 0$. Using this notation we prove the following elementary counting result.

\begin{lemma}
\label{basiclemma}
If the point graph of a partial geometry $pg(s,t,\alpha)$ is $n$-e.c., then 
\begin{equation*}
n \leq \text{min}\{s-\alpha+1, t+1, \alpha+1\}.
\end{equation*}
\end{lemma}

\begin{proof}
Suppose $n \geq s-\alpha+2$. Let $B$ denote an $n$-set containing $s-\alpha+2$ vertices incident to a common line $L$ in the partial geometry. Every other vertex is adjacent to at least $\alpha$ of the $s$ vertices on $L$, and so every vertex in the point graph is adjacent to at least one vertex of $B$. Thus $\Gamma(\emptyset, B) = 0$.

Next suppose $n \geq t+2$. Choose a vertex $x$, and let $B$ denote an $(n-1)$-set containing $t+1$ points, one from each of the $t+1$ lines incident to $x$. Every vertex that is adjacent to $x$ in the point graph must also be adjacent to at least one point in $B$. Thus $\Gamma(\{x\},B) = 0$.

Finally suppose that $n \geq \alpha + 2$. Let $A$ denote an $(n-1)$-set containing $\alpha + 2$ vertices on $L$. Each vertex in the point graph is either on $L$ or adjacent to exactly $\alpha$ vertices of $L$. Therefore there is no vertex adjacent to exactly $\alpha + 1$ vertices of $L$, and so for $x$ in $A$ we have $\Gamma(A,\{x\}) = 0$. \qed
\end{proof}

As an immediate consequence of this lemma, we see that the point graph of a generalized quadrangle is not $3$-e.c.~and is therefore not $n$-e.c for $n \geq 3$. To eliminate further classes of geometric graphs, we make the following observation.

\begin{lemma}
\label{mubound}
Let $G$ denote a $3$-e.c.~graph with vertex set $V$, and let $N(x)$ and $N(y)$ denote the neighbourhoods of two distinct vertices $x$ and $y$. Then every vertex in $V \setminus (N(x) \cup N(y) \cup \{x,y\})$ is adjacent to at least one vertex in $N(x) \cap N(y)$.
\end{lemma}
\begin{proof}
Let $z$ denote a vertex in $V \setminus (N(x) \cup N(y) \cup \{x,y\})$. Since $G$ is $3$-e.c.~there is a vertex adjacent to $x$, $y$, and $z$, and it must lie in the set $N(x) \cap N(y)$. \qed
\end{proof}

This yields the following inequality for strongly regular graphs.

\begin{corollary}
If $G$ is $3$-e.c.~and strongly regular with parameters $(v,k,\lambda,\mu)$, then 
\begin{equation*}
\mu (k-\lambda-3) \geq v - 2k + \mu - 2.
\end{equation*}
\end{corollary}
\begin{proof}
Let $V$ denote the vertex set of $G$, and let $x$ and $y$ be two nonadjacent vertices. Then $|N(x) \cap N(y)| = \mu$. Let $z$ denote a vertex in $N(x) \cap N(y)$, and note that $z$ is adjacent to at most $k-\lambda$ vertices outside of $N(x) \cup N(y) \cup \{x,y\}$. In particular, since $G$ is $3$-e.c., the vertex $z$ must be adjacent to at least one vertex in $N(x) \setminus N(y)$ and at least one other vertex in $N(y) \setminus N(x)$. Therefore each $z$ is adjacent to at most $(k-\lambda-3)$ vertices outside of $N(x) \cup N(y) \cup \{x,y\}$. Finally we note that 
\begin{equation*}
|V \setminus (N(x) \cup N(y) \cup \{x,y\})| = v - 2k + \mu - 2,
\end{equation*}
and so the above inequality follows from Lemma \ref{mubound}. \qed
\end{proof}

For example, this inequality implies that the Shrikhande graph with parameters $(16,6,22)$ is not $3$-e.c. (However, this graph can also be eliminated with more elementary reasoning.) It turns out to be more useful to consider Lemma \ref{mubound} applied to geometric graphs, since we can obtain a tighter bound by having more information about the clique structure of the graph.

\begin{corollary}
\label{bettercounting}
If $G$ is a $3$-e.c.~point graph of a partial geometry with parameters $(s,t,\alpha)$ such that $s \geq 2 \alpha - 1$, then 
\begin{equation*}
\alpha (t^2-1)(s - 2\alpha + 2) \geq (s+1)(st+\alpha)/\alpha + (\alpha - 2s)(t+1) - 2.
\end{equation*}
\end{corollary}
\begin{proof}
Let $x$ and $y$ be two nonadjacent vertices in the vertex set of $G$. Let $P = |N(x) \cap N(y)|$, and let $T$ denote the number of vertices in $V \setminus (N(x) \cup N(y) \cup \{x,y\})$ that are adjacent to at least one vertex in $N(x) \cap N(y)$. Note that $|P| = \alpha(t+1)$ and each vertex in $P$ is incident to $t-1$ lines that are not incident to $x$ or $y$. Let $\mathcal{L}$ denote the set of all lines incident to at least one point in $P$ but not incident to $x$ or $y$. For each line $l$ in $\mathcal{L}$, let $m_l$ denote the number of points incident to $l$ that are contained in $P$. Furthermore we write $l \sim p$ if the line $l$ is incident to the point $p$. Using this notation we see that
\begin{equation*}
|T| = \sum_{l \in \mathcal{L}} (s+1) - (2\alpha - m_l) = \sum_{p \in P} \sum_{\substack{l \in \mathcal{L} \\ l \sim p}}  \frac{(s + 1) - (2 \alpha - m_l)}{m_l}.
\end{equation*}
Moreover, since we assume $s \geq 2 \alpha + 1$, we see 
\begin{equation*}
\frac{(s + 1) - (2 \alpha - m_l)}{m_l} \leq s- 2\alpha + 2
\end{equation*}
for all possible values of $m_l$ between $1$ and $s+1$. Therefore we have
\begin{equation*}
|T| \leq \sum_{p \in P} \sum_{\substack{l \in \mathcal{L} \\ l \sim p}} s- 2\alpha + 2 = \alpha(t^2-1)(s-2\alpha + 2).
\end{equation*}
However, by Lemma \ref{mubound} we must have $T = V \setminus (N(x) \cup N(y) \cup \{x,y\})$. By a straightforward computation we have
\begin{equation*}
|(N(x) \cup N(y) \cup \{x,y\})| = (s+1)(st+\alpha)/\alpha + (\alpha - 2s)(t+1) - 2,
\end{equation*} 
and combining this with our earlier upper bound on $|T|$, we obtain our desired inequality. \qed
\end{proof}

Note that the previous result only holds if $s$ is sufficiently large with respect to $\alpha$. If we have $s \leq 2 \alpha - 2$, then we have the following weaker result.

\begin{corollary}
\label{weakcounting}
If $G$ is a $3$-e.c.~point graph of a partial geometry with parameters $(s,t,\alpha)$, then 
\begin{equation*}
\alpha (t^2-1)(s - \alpha + 1) \geq (s+1)(st+\alpha)/\alpha + (\alpha - 2s)(t+1) - 2.
\end{equation*}
\end{corollary}
\begin{proof}
The proof technique is the same as in the proof of Corollary \ref{bettercounting}, with the exception that we are only able to obtain the weaker bound
\begin{equation*}
|T| \leq \sum_{l \in \mathcal{L}} (s+1) - (2\alpha - m_l) = \alpha(t^2-1)(s-\alpha + 1).
\end{equation*} \qed
\end{proof}

\subsection{Nets}
Recall that a net is a partial geometry with parameters $(s,t,t)$. For example, the Paley graph on $q^2$ vertices is the point graph of a net with parameters $(q-1,(q-1)/2,(q-1)/2)$. The affine plane graphs considered by Baker, Bonato, Brown, and Sz{\H{o}}ny are examples of points graphs of nets \cite{bonato_affine}. In particular, they prove that the point graphs of partial geometries obtained from Desarguesian planes on at least 64 points are always $3$-e.c. We focus instead on describing general necessary conditions on parameters of nets that are $3$-e.c.

Since we have already seen that generalized quadrangles do not give us $3$-e.c.~graphs, we turn our attention to partial geometries with the parameters $(s,2,2)$. These partial geometries correspond to Latin squares, and the resulting strongly regular graphs have parameters $(s^2,3(s-1),s,6)$. By Lemma \ref{basiclemma} we see that if the point graph of $pg(s,2,2)$ is $n$-e.c., then $n \leq 3$. Applying our earlier bound and eliminating some small cases, we have the following result.

\begin{theorem}
There is a unique $3$-e.c.~graph that is the point graph of a partial geometry with parameters $(s,2,2)$.
\end{theorem}
\begin{proof}
Let $G$ denote a $3$-e.c.~point graph of a partial geometry $pg(s,2,2)$. By Lemma \ref{basiclemma} we see that $s \geq 4$.  Therefore Corollary \ref{bettercounting} applies, and we see that
\begin{equation*}
6(s-2) \geq (s+1)^2 - 6s + 4.
\end{equation*}
Thus we see that $8 \geq (s-5)^2$, and since $s$ is an integer we must have $s \leq 7$.
Let $H$ denote the induced subgraph of the neighbourhood of a vertex $x$ in $G$. Note that the vertices of $H$ can be partitioned into $3$ disjoint cliques of size $s-1$, say $C_1$, $C_2$, and $C_3$. Without loss of generality, every vertex in $C_1$ is adjacent to exactly one vertex in $C_2$ and exactly one vertex in $C_3$. Since $G$ is $3$-e.c., every pairs of adjacent vertices $u$ and $v$ with $u \in C_1$ and $v \in C_2$ must have a common neighbour $w$ in $C_3$. Therefore the edges in $H$ that are not contained in $C_1$, $C_2$, or $C_3$ can be partitioned into edge- and vertex-disjoint triangles. This implies that the corresponding Latin square satisfies the so-called hexagonal condition \cite{jungnickel}. Note that every Latin square can be constructed from the Cayley table of a quasigroup. In particular, since the graph corresponding to a Latin square is invariant under permutation of rows and columns of the Latin square, we see that every such graph corresponds to the Cayley table of a quasigroup with an identity element. Quasigroups with identity elements are called loops. Therefore $pg(s,2,2)$ can be constructed from the Cayley table of a loop in which every element is its own inverse. A loop with this property satisfies the hexagonal condition if and only if it an abelian group. In particular, it must be the additive group of $\ints_2^t$ for some positive integer $t$. Thus $s+1 = 2^t$ for some positive integer $t$. Since $5 \leq s+1 \leq 8$, we conclude that $s = 7$ and $G$ is the point graph of the net constructed from the Cayley table of $\ints_2^3$.
 \qed
\end{proof}

\begin{theorem}
If $G$ is a $3$-e.c.~point graph of a partial geometry with parameters $(s,t,t)$ with $s \geq 2\alpha - 1$ then
\begin{equation*}
s^2 - (t^3+t)s + 2t^4 - 2t^3 - t^2 + 3t - 1 \leq 0.
\end{equation*}
\end{theorem}
\begin{proof}
Substituting $\alpha = t+1$ into Corollary \ref{bettercounting} yields the above result. \qed
\end{proof}

Note that for a fixed $t$ this gives an upper bound on $s$ such that the point graph of $(s,t,t)$ is possibly $3$-e.c. 

\subsection{Dual Designs}
First we consider the point graph of a dual of a $2$-$(v,3,1)$ design, which is the block-intersection graph of a Steiner triple system. Note that the dual of a $2$-$(v,3,1)$ design is a partial geometry with parameters $(s,2,3)$ where $s+1 = (v-1)/2$. The following result is due to Griggs, Grannel, and Forbes \cite{forbes}. Their proof requires additional results about Steiner triple systems.

\begin{theorem}
If $G$ is a $3$-e.c.~point graph of a partial geometry with parameters $(s,2,3)$, then $s = 8 \; \text{or} \; 9$. \qed
\end{theorem}

Using our general bound we obtain a weaker result.

\begin{theorem}
If $G$ is a $3$-e.c.~point graph of a partial geometry with parameters $(s,2,3)$, then $s \leq 15$.
\end{theorem}
\begin{proof}
Let $G$ denote the point graph of a partial geometry $pg(s,2,3)$. By Lemma \ref{basiclemma}, we see that $s \geq 5$. Therefore Corollary \ref{bettercounting} applies, and we must have
\begin{equation*}
9(s-4) \geq \frac{2}{3}s^2 - \frac{13}{3}s + 18.
\end{equation*} 
Simplifying this expression and completing the square yields
\begin{equation*}
\frac{43}{2} \geq (s-10)^2.
\end{equation*}
Since $s$ is a positive integer, it must be the case that $s \leq 16$. However, recall that the dual of a partial geometry $(s,2,3)$ corresponds to a Steiner triple system on $v = 2s+3$ points. The above bound on $s$ implies that $v \leq 35$. It is well-known a Steiner triple system on $v$ points exists if and only if $v \equiv 1 \pmod 6$ or $v \equiv 3 \pmod 6$ (See Chapter 19 \cite{vanlint}.) This implies $v \leq 33$, and so $s \leq 15$. \qed
\end{proof}

Despite a weaker bound for Steiner triple systems, our result generalizes to give bounds for the block intersection graphs of $2$-$(v,t+1,1)$ designs for larger $t$. (These designs are referred to as Steiner systems.) Note that the dual of a $2$-$(v,t+1,1)$ designs is partial geometry with parameters $(s,t,t+1)$ where $s+1 = (v-1)/(t-1)$. By another application of Corollary \ref{bettercounting}, we have the following result.

\begin{theorem}
If $G$ is a $3$-e.c.~point graph of a partial geometry with parameters $(s,t,t+1)$ with $s \geq  2t + 1$ then
\begin{equation*}
s^2 - (t^3+2t^2+2t)s + 2t^4 + 4t^3 + t^2 -t \leq 0.
\end{equation*}
\end{theorem}
\begin{proof}
Substituting $\alpha = t+1$ into Corollary \ref{bettercounting} yields the above result. \qed
\end{proof}

The previous result implies a concrete upper bound on $s$ in terms of $t$ such that the point graph of a partial geometry $(s,t,t+1)$ is $3$-e.c. A natural progression from Griggs, Grannel, and Forbes' work would be to search for $4$-e.c. graphs that are point graphs of partial geometries with parameters $(s,3,4)$.  In this case, the previous result implies that any such partial geometry must satisfy $s \leq 44$.

\bibliography{nec_paper}
\bibliographystyle{plain}
\end{document}